\documentclass[a4paper]{amsart}

\usepackage{hyperref}
\usepackage{fullpage}
\usepackage{ amsmath,amstext,amsthm,amscd,amsopn,verbatim,amssymb, amsfonts}
\usepackage{MnSymbol}

\usepackage[bbgreekl]{mathbbol}
\usepackage{array}
\usepackage{tikz, tikz-cd}
\usetikzlibrary{matrix}
\usetikzlibrary{shapes}
\usetikzlibrary{arrows}
\usetikzlibrary{calc,3d}
\usetikzlibrary{decorations,decorations.pathmorphing}
\usetikzlibrary{through}
\tikzset{ext/.style={circle, draw,inner sep=1pt, minimum size=5pt},
         int/.style={circle,draw,fill,inner sep=1pt, minimum size=5pt},nil/.style={inner sep=1pt}}
\tikzset{exte/.style={circle, draw,inner sep=3pt},inte/.style={circle,draw,fill,inner sep=3pt}}
\tikzset{diagram/.style={matrix of math nodes, row sep=3em, column sep=2.5em, text height=1.5ex, text depth=0.25ex}}
\tikzset{diagram2/.style={matrix of math nodes, row sep=0.5em, column sep=0.5em, text height=1.5ex, text depth=0.25ex}}
\theoremstyle{plain}
  \newtheorem{thm}{Theorem}
  \newtheorem{defi}[thm]{Definition}
  \newtheorem{prop}[thm]{Proposition}

\theoremstyle{definition}
  
  \newtheorem{rem}[thm]{Remark}

\newcommand{\R}{{\mathbb{R}}}

\newcommand{\Graphs}{{\mathsf{Graphs}}}
\newcommand{\fGraphs}{{\mathsf{fGraphs}}}

\newcommand{\SGraphs}{{\mathsf{SGraphs}}}

\newcommand{\fdGraphs}{{\mathsf{fdGraphs}}}

\newcommand{\Gra}{{\mathsf{Gra}}}

\newcommand{\dGra}{{\mathsf{dGra}}}

\newcommand{\KGra}{{\mathsf{Gra}}^1}

\newcommand{\KGraphs}{{\mathsf{Graphs}}^1}

\newcommand{\Poiss}{{\mathsf{Pois}}}
\newcommand{\Pois}{\Poiss}

\newcommand{\PA}{{\mathit{PA}}}

\newcommand{\gra}{{\mathrm{gra}}}

\newcommand{\kgra}{{\mathrm{kgra}}}

\newcommand{\SGra}{{\mathsf{SGra}}}

\newcommand{\K}{\mathbb{R}}

\newcommand{\op}{\mathcal}

\newcommand{\Br}{\mathsf{Br}}

\newcommand{\Lie}{\mathsf{Lie}}
\newcommand{\ELie}{\mathsf{ELie}}

\newcommand{\hoLie}{\mathsf{hoLie}}

\newcommand{\fSGraphs}{\mathsf{fSGraphs}}
\newcommand{\SC}{\mathsf{SC}}

\newcommand{\Ass}{\mathsf{Assoc}}

\newcommand{\bigGraphs}{\mathsf{bigGraphs}}
\newcommand{\FM}{\mathsf{FM}}
\newcommand{\uFM}{\mathsf{uFM}}
\newcommand{\uSFM}{\mathsf{uSFM}}
\newcommand{\uFMH}{\uFM^1}

\newcommand{\bpm}{\begin{pmatrix}}
\newcommand{\epm}{\end{pmatrix}}
\newcommand{\Tpoly}{T_{\rm poly}}
\newcommand{\Dpoly}{D_{\rm poly}}

\newcommand{\hoPoiss}{\mathsf{hoPois}}
\newcommand{\hoPois}{\mathsf{hoPois}}

\newcommand{\mU}{\mathcal{U}}

\newcommand{\LD}{\mathsf{D}}
\newcommand{\lD}{\LD}

\newcommand{\FMH}{\mathsf{FM}^1}
\newcommand{\SFM}{\mathsf{SFM}}
\newcommand{\e}{\mathsf{e}}
\newcommand{\KGraphsM}{\mathsf{KGraphsM}}
\newcommand{\uGraphs}{\mathsf{uGraphs}}
\newcommand{\uKGraphs}{\uGraphs^1}
\newcommand{\uSGraphs}{\mathsf{uSGraphs}}
\newcommand{\bo}{\mathbf{1}}
\newcommand{\EGer}{\mathsf{EGer}}
\newcommand{\aor}{\rcirclearrowleft}
\newcommand{\aol}{\lcirclearrowright}

\newcommand{\Conf}{\mathrm{Conf}}

\newcommand{\fKGraphs}{\fGraphs^1}

\begin{document}
\title{Models for the $n$-Swiss Cheese operads}
\author{Thomas Willwacher}
\address{Institute of Mathematics\\ University of Zurich\\ Winterthurerstrasse 190 \\ 8057 Zurich, Switzerland}
\email{thomas.willwacher@math.uzh.ch}

\thanks{The author was partially supported by the Swiss National Science foundation, grant 200021-150012, and the SwissMAP NCCR funded by the Swiss National Science foundation.}


\begin{abstract}
We describe combinatorial Hopf (co-)operadic models for the Swiss Cheese operads built from Feynman diagrams. This extends previous work of Kontsevich and Lambrechts-Voli\'c for the little disks operads.
\end{abstract}
\maketitle

\section{Introduction}
The little $n$-disks operads $\LD_n$ are operads of rectilinear embeddings of ``small'' $n$-dimensional disks in the $n$-dimensional unit disk. The operadic composition is obtained by gluing one such configuration into one of the small disks, cf. Figure \ref{fig:LDcomposition} below.
Similarly, one may consider spaces $\LD^1_n$ of rectilinear embedings of ``small'' disks and half-disks in the $n$-dimensional unit half-disk, such that the equators of the small half disks sit in the equatorial plane of the unit half-disk.
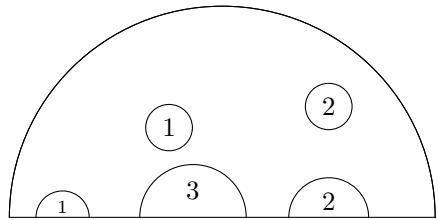
\[
\begin{tikzpicture}[scale=.7]
 \clip (-5,-.008) rectangle (5,5);
 \draw (0,0) circle (4);
 \node[draw, circle] at (-1,1.7) {1};
 \node[draw, circle, minimum width=10pt] at (2,2.1) {2};
 \draw (0,0) circle (4);
 \draw(-4,0) -- (4,0);
 \draw (-3,0) circle (.5);
 \node at (-3,0.2) {$\scriptstyle 1$};
 \draw (-.55,0) circle (1);
 \node at (-.55,0.5) {$3$};
 \draw (2,0) circle (.75);
 \node at (2,0.3) {$2$};
\end{tikzpicture}
\]
By gluing such configurations, one can can naturally define a two-colored operad structure on the pair $\SC_n=(\LD_n,\LD_n^1)$, which is called the $n$-dimensional Swiss Cheese operad, cf. \cite{voronov_swisscheese}.
The little disks and Swiss-Cheese operads play a central role in several disciplines of mathematics. In particular, understanding their rational homotopy theory is important for computing the factorization homology of manifolds with boundary \cite{AyalaFrancis}, and in the embedding calculus for such manifolds \cite{GoodwillieWeiss}.

The goal of this short note is to describe real combinatorial (cofibrant) dg Hopf cooperad models for the topological colored operads $\SC_n$. We expect that the result will have applications in the aforementioned areas of mathematics.

In fact, we will work with versions of the Swiss Cheese operads built from configuration spaces of points rather than configurations of disks as in \cite{voronov_swisscheese}.
Concretely, the topological operad $\FM_n$ is defined as a compactification of the space of configurations of finite sets of points in $\R^n$, modulo scaling and translation. It is homotopic to $\LD_n$, see \cite{K2}. Similarly, the (homotopic) variant of $\LD_n^1$ are configurations of points in the upper halfplane and on the real axis, modulo scaling and translation, suitably compactified so that the operadic compositions are defined \cite{voronov_swisscheese}. We denote these configuration spaces collectively by $\FM_n^1$, and the two-colored topological operad formed by 
\[
 \SFM_n=(\FM_n,\FM_n^1).
\]
In fact, $\SFM_n$ is an operad of semi-algebraic manifolds, and in each fixed arity we may consider the dg algebras of PA forms \cite{HLTV} on the underlying configuration spaces. We denote this collection by $\Omega_{\PA}(\SFM_n)$. 

The main result of this paper is to describe a combinatorial 2-colored differential graded Hopf\footnote{A Hopf cooperad is a cooperad in the category of differential graded commutative algebras, cf. \cite{lodayval}.} cooperad $\SGraphs_n=(\Graphs_n, \KGraphs_n)$ with a map (of colored collections) $\SGraphs_n\to \Omega_{\PA}(\SFM_n)$ that satisfies the following properties:
\begin{itemize}
 \item In each arity the map is a quasi-isomorphism of differential graded commutative algebras.
 \item The map is compatible with the (co)operadic (co)compositions in the sense that the following diagrams commute:
 \begin{equation}\label{equ:opcompat1}
  \begin{tikzcd}
   \Graphs_n(r) \ar{r} \ar{dd} & \Omega_{\PA}(\FM_n(r)) \ar{d} \\
   & \Omega_{\PA}(\FM_n(r-k+1)\times \FM_n(k)) \\
   \Graphs_n(r-k+1)\otimes \Graphs_n(k) \ar{r} & \Omega_{\PA}(\FM_n(r-k+1))\otimes \Omega_{\PA}(\FM_n(k)) \ar{u}
  \end{tikzcd}
 \end{equation}
 \begin{equation}\label{equ:opcompat2}
  \begin{tikzcd}
   \Graphs_n^1(r,s) \ar{r} \ar{dd} & \Omega_{\PA}(\FM_n^1(r,s)) \ar{d} \\
   & \Omega_{\PA}(\FM_n^1(r-k, s-l+1)\times \FM_n^1(k,l)) \\
   \Graphs_n^1(r-k,s-l+1)\otimes \Graphs_n^1(k,l) \ar{r} & \Omega_{\PA}(\FM_n^1(r-k, s-l+1))\otimes \Omega_{\PA}(\FM_n^1(k,l)) \ar{u}
  \end{tikzcd}
 \end{equation}
 \begin{equation}\label{equ:opcompat3}
  \begin{tikzcd}
   \Graphs_n^1(r,s) \ar{r} \ar{dd} & \Omega_{\PA}(\FM_n^1(r,s)) \ar{d} \\
   & \Omega_{\PA}(\FM_n^1(r-k+1,s)\times \FM_n(k)) \\
   \Graphs_n^1(r-k+1,s)\otimes \Graphs_n(k) \ar{r} & \Omega_{\PA}(\FM_n^1(r-k+1,s))\otimes \Omega_{\PA}(\FM_n(k)) \ar{u}
  \end{tikzcd}
 \end{equation}
 \item The graded commutative algebras underlying $\SGraphs$ are free.
\end{itemize}

Given the above properties, we consider $\SGraphs_n$ to be a dg commutative algebra model for the topological Swiss Cheese operads $\SFM_n$ and $\SC_n$.

In fact, our construction is merely an extension to the Swiss-Cheese setting of a similar construction of Kontsevich \cite{K2} for the little $n$-disks operads, which was implemented in detail and more rigor by Lambrechts and Voli\'c \cite{LV}.
Our contribution in this paper is to solve some technical issues which before have prevented this extension, and whose solution the author had missed in \cite{graphthings}.

To describe the model, let us quickly summarize Kontsevich's construction.\footnote{We present a slightly non-standard (and slightly too complicated) approach to Kontsevich's construction. This approach will however set the stage for the Swiss Cheese case later. }
For this one considers a cooperad $\fGraphs_n$ whose operations in arity $r$ are linear combinations of graphs with $r$ numbered (``external'') vertices and an arbitrary (finite) number of unidentifiable ``internal'' vertices like the following:
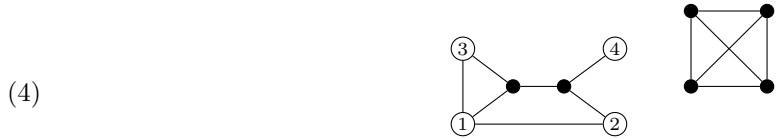
\begin{equation}\label{equ:graphsexample}
 \begin{tikzpicture}[baseline=2ex]
  \node[ext] (v1) at (0,0) {$\scriptstyle 1$};
  \node[ext] (v2) at (2,0) {$\scriptstyle 2$};
  \node[ext] (v3) at (0,1) {$\scriptstyle 3$};
  \node[ext] (v4) at (2,1) {$\scriptstyle 4$};
  \node[int] (w1) at (.66,.5) {};
  \node[int] (w2) at (1.33,.5) {};
  \node[int] (w3) at (3,.5) {};
  \node[int] (w4) at (4,.5) {};
  \node[int] (w5) at (3,1.5) {};
  \node[int] (w6) at (4,1.5) {};
  \draw (v1) edge (v2) edge (v3) edge (w1)
        (w1) edge (w2) edge (v3)
        (w2) edge (v2) edge (v4)
        (w3) edge (w4) edge (w5) edge (w6)
        (w4) edge (w5) edge (w6)
        (w5) edge (w6);
 \end{tikzpicture}
\end{equation}

For each such graph $\Gamma\in\fGraphs(r)$ (with, say, $k$ internal vertices) one may define a PA form $\omega_\Gamma\in \Omega_{\PA}(\FM_n(r))$.
This is done by a fiber integral
\begin{equation}\label{equ:integral1}
  \omega_\Gamma = \int_{\FM_n(k+r)\to \FM_n(r)} \bigwedge_{(i,j)\text{ edge}} \pi_{ij}^*\Omega_{S^{n-1}},
\end{equation}
where the integral is taken over the fiber of the forgetful map $\FM_n(k+r)\to \FM_n(r)$, and the integrand is obtained as a product of pullbacks of the round volume form on the sphere by the forgetful maps 
\[
 \pi_{ij} : \FM_n(k+r) \to \FM_n(2) =S^{n-1}
\]
forgetting all but two points.

The spaces $\fGraphs_n(r)$ naturally assemble into a dg Hopf cooperad $\fGraphs_n$, and it can be checked using Stokes' Theorem that the map $\fGraphs_n\to \Omega_{\PA}(\FM_n)$ is compatible with all algebraic operations, in particular the cooperad structure in the sense of \eqref{equ:opcompat1}.
However, the maps $\fGraphs_n(r)\to \Omega_{\PA}(\FM_n(r))$ are not quasi-isomorphisms (yet). 
To have a quasi-isomorphism, one has to pass to the connected quotient. We call a graph in $\fGraphs_n(r)$ \emph{externally disconnected} if it has a connected component containing only internal vertices but no external. For example, the graph in the above picture is externally disconnected.
One then has the following result.
\begin{thm}[Kontsevich {\cite{K2}}, Lambrechts-Voli\'c{\cite{LV}}] \label{thm:KLV}
 The space of externally disconnected $\fGraphs_n^{disc}$ is a bi-ideal, so that the quotient collection 
 \[
  \Graphs_{n} = \fGraphs_n / \fGraphs_n^{disc}
 \]
 inherits a dg Hopf cooperad structure.
 
 The map $\fGraphs_n\to \Omega_{\PA}(\FM_n)$ factors through $\Graphs_n$, and the map 
 \[
  \Graphs_n \to \Omega_{\PA}(\FM_n)
 \]
is a quasi-isomorphism compatible with all algebraic structures.
\end{thm}

In this paper, we extend the Kontsevich construction as follows. To model the spaces $\FM_n^1(r,s)$ of configurations of $r$ points in the halfspace and $s$ on the boundary, we consider the dg commutative algebra $\fGraphs_n^1(r,s)$, whose elements are linear combinations of directed graphs with two types of vertices (usually denoted by type I and type II), with $r$ type one vertices and $s$ type II vertices numbered (``external''), and the remainder of vertices unidentifiable (``internal''), as for example in the following picture.
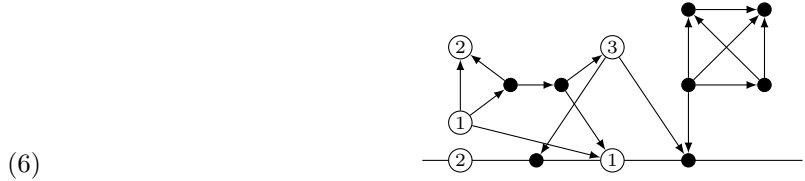
\begin{equation}\label{equ:kgraphsexample}
  \begin{tikzpicture}
  \draw (-.5,-.5) -- (4.5,-.5);
  \node[ext] (v1) at (0,0) {$\scriptstyle 1$};
  \node[ext, fill=white] (v2) at (2,-.5) {$\scriptstyle 1$};
  \node[ext] (v3) at (0,1) {$\scriptstyle 2$};
  \node[ext] (v4) at (2,1) {$\scriptstyle 3$};
  \node[ext, fill=white] (v5) at (0,-.5) {$\scriptstyle 2$};
  \node[int] (w1) at (.66,.5) {};
  \node[int] (w2) at (1.33,.5) {};
  \node[int] (w3) at (3,.5) {};
  \node[int] (w4) at (4,.5) {};
  \node[int] (w5) at (3,1.5) {};
  \node[int] (w6) at (4,1.5) {};
  \node[int] (u1) at (1,-.5) {};
  \node[int] (u2) at (3,-.5) {};
  \draw[-latex] (v1) edge (v2) edge (v3) edge (w1)
        (w1) edge (w2) edge (v3)
        (w2) edge (v2) edge (v4)
        (w3) edge (w4) edge (w5) edge (w6)
        (w4) edge (w5) edge (w6)
        (w5) edge (w6)
        (v4) edge (u1) edge(u2)
        (w3) edge (u2);       
 \end{tikzpicture}
\end{equation}
For $n=2$ we require in addition that all type II vertices (internal and external) are equipped with a linear ordering, which we indicate by drawing these vertices on a line as above.

Now, to each such graph $\Gamma\in \fGraphs_n^1(r,s)$ (with say $k$ and $l$ internal vertices of types I and II) we may associate a differential form $\omega^1_\Gamma\in\Omega_{\PA}(\FM_n^1(r,s))$ by the following variant of \eqref{equ:integral1} above:
\begin{equation}\label{equ:integral2}
 \omega_\Gamma^1 = \int_{\FM_n^1(r+k,s+l)\to\FM_n^1(r,s)}
 \bigwedge_{(i,j)\text{ edge}}
 \pi_{ij}^*\Omega_{S^{n-1}}.
\end{equation}
Here the integral is again over the fiber of the forgetful map. The integrand is again a product of pullbacks of the round volume form on the sphere along the forgetful maps 
\[
 \pi_{ij}: \FM_n^1(r+k,s+l) \to 
 \begin{cases}
  \FM_n^1(2,0) &\text{for $i,j$ of type I} \\
  \FM_n^1(1,1) &\text{for $i$ of type I, $j$ of type II}
 \end{cases}
 \to S^{n-1},
\]
where the final projection is by projecting the $j$-th point into the unit sphere in the tangent space at the $i$-th point, along the hyperbolic geodesic, as indicated in the following picture (for 2 dimensions).
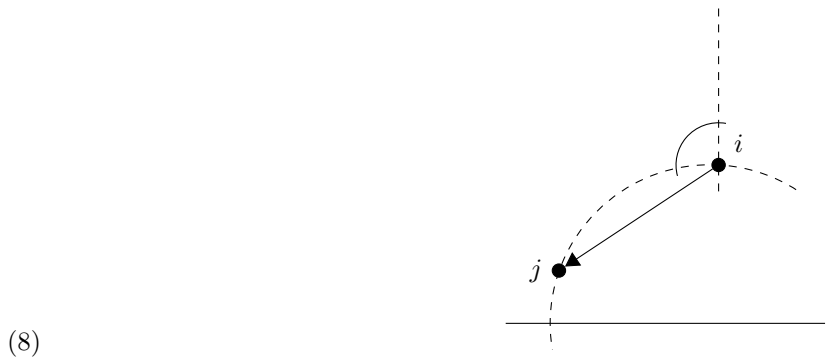
\begin{equation}\label{equ:hyperbolicanglepicture}
\begin{tikzpicture}[scale=.7]
\draw (-3,0)--(3,0);
\node[int, label=30:{$i$}] (v1) at (1,3) {};
\node[int, label=180:{$j$}] (v2) at (-2,1) {};
\draw[dashed] (v1)+(0,-.5) -- +(0,3);
\draw (v1)+(80:.8) arc (80:195:.8);
\clip (-5,-.5) rectangle (2.5,4);
\node [draw, dashed] at (.84,0) [circle through={(v1)}] {};
\draw[-triangle 60] (v1)--(v2);
\end{tikzpicture}
\end{equation}
Analogously to the little-disks setting, one may again endow the pair 
\[
\fSGraphs_n=(\fGraphs_n, \fGraphs_n^1) 
\]
with a two-colored dg Hopf operad structure, such that the maps
\[
 \fSGraphs_n\to \Omega(\SFM_n)
\]
are compatible with all algebraic structures.
Continuing in analogy with the little disks setting, we wish to (i) define a bi-ideal of ``externally disconnceted'' graphs $\fSGraphs_n^{disc}$ such that the map $\fSGraphs_n\to \Omega(\SFM_n)$ factors through the quotient, and such that (ii) this map on the quotient is a quasi-isomorphism.
We note that step (i) is not obvious if $n=2$, since the naive definition of having connected components with only internal vertices does not work. Computing the cohomology in step (ii) is a bit more complicated than in the little disk setting, owed in part to the fact that the Swiss-Cheese operads are not formal \cite{livernet}.
The contribution of this note is to solve the problems (i) and (ii) and show the following result.
\begin{thm}\label{thm:main}
 There is a natural Hopf cooperadic bi-ideal $\fSGraphs_n^{disc}\subset \fSGraphs_n$ such that the map $ \fSGraphs_n\to \Omega(\SFM_n)$ factors through the quotient dg Hopf cooperad
 \[
 \SGraphs_n =    \fSGraphs_n / \fSGraphs_n^{disc},
 \]
and such that the induced map $\SGraphs_n\to \Omega(\SFM_n)$ is a quasi-isomorphism compatible with the differential graded commutative algebra structures and the (co)operadic (co)compositions as in \eqref{equ:opcompat1}-\eqref{equ:opcompat3}.
\end{thm}
Furthermore, one can easily check that all arity components of $\SGraphs_n$ are free as graded commutative algebras, and thus $\SGraphs_n$ forms the desired dgca model (in the sense described above) for the Swiss-Cheese operad $\SC_n$.

The remainder of this paper is organized as follows: Section \ref{sec:notation} is devoted to a somewhat more careful recollection/definition of the notation, in particular of the graphical models sketched above.
Section \ref{sec:ideal} discusses the bi-ideal $\fSGraphs_n^{disc}$.
The cohomology computation for the quotient that finishes the proof of Theorem \ref{thm:main} is contained in section \ref{sec:cohom}.
Finally, in section \ref{sec:application} we present an application of our result to deformation quantization: We show that any $L_\infty$ stable formality morphism may be extended (up to homotopy, see Theorem \ref{thm:defq} below) to a homotopy Gerstenhaber formality morphism.

\begin{figure}
 \[
\begin{tikzpicture}[baseline=-.65ex,yshift=-1cm]
\draw (1,1) circle (1.2);
\node[draw,circle, minimum height=.5cm,minimum width=.5cm] at (.5,.5) {1};
\node[draw,circle,minimum height=.5cm,minimum width=.5cm] at (.7,1.5) {2};
\node[draw,circle, minimum height=.9cm,minimum width=.9cm] at (1.33,.65) {3};
\end{tikzpicture}
\circ_3
\begin{tikzpicture}[baseline=-.65ex]
\draw (0,0) circle (1.2);
\node[draw,circle, minimum height=.8cm,minimum width=.8cm] at (-.5,0) {a};
\node[draw,circle, minimum height=.8cm,minimum width=.8cm] at (.5,0) {b};
\end{tikzpicture}
=
\begin{tikzpicture}[baseline=-.65ex,yshift=-1cm]
\draw (1,1) circle (1.2);
\node[draw,circle, minimum height=.5cm,minimum width=.5cm] at (.5,.5) {1};
\node[draw,circle,minimum height=.5cm,minimum width=.5cm] at (.7,1.5) {2};
\begin{scope}[xshift=1.43cm,yshift=.65cm, scale=.5]
\draw[dotted] (0,0) circle (1.2);
\node[draw,circle, minimum height=.1cm,minimum width=.1cm] at (-.5,0) {$\scriptstyle a$};
\node[draw,circle, minimum height=.1cm,minimum width=.1cm] at (.5,0) {$\scriptstyle b$};
\end{scope}
\end{tikzpicture}
\]
\caption{\label{fig:LDcomposition} An example for the composition (gluing) in the little disks operad. }
\end{figure}
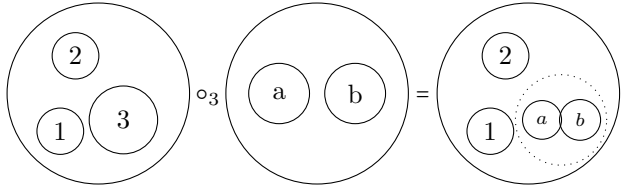

\subsection*{Acknowledgements}
I am grateful for helpful discussions with Anton Khoroshkin and Victor Turchin. In particular, Anton and Sergei Merkulov have independently obtained (yet unpublished) results in the direction of those of the present paper.

\section{Notation}\label{sec:notation}
We generally work over the ground field $\R$. All vector spaces and differential graded vector spaces considered are $\R$-vector spaces.
We use cohomological conventions for all differential graded (dg) vector spaces we consider, i.e., the differential always has degree $+1$ under all circumstances. In particular, this means that, for example, the homology of a topological space is considered to be concentrated in non-positive degrees. 

We use the language of operads and colored operads, and mostly follow the conventions of the textbook by Loday and Vallette \cite{lodayval}, to which we frequently appeal.
Operads of particular importance for us are the associative operad $\Ass$, the Lie operad $\Lie$, and the $n$-Poisson operad $\Pois_n$, generated by two binary operations of degree 0 and $1-n$. Finally, we will need the braces operad $\Br$, cf. \cite{graphthings} for the definition we will use here.

We denote the $k$-fold (de)suspension of an operad $\op P$ by $\op P\{k\}$, cf. \cite{lodayval}. In particular, we denote $\Lie_k:=\Lie\{k-1\}$, so that one has a map $\Lie_n\to \Pois_n$.

At several (few) places we will use quasi-free resolutions of these operads. In particular, we denote by 
\[
 \hoLie_2 =\Omega(\Lie_2^\vee)
\]
the operadic cobar construction of the Koszul dual cooperad $\Lie_2^\vee$ of $\Lie_2$.
Similarly, we define $\hoPois_2 =\Omega(\Pois_2^\vee)$.

\subsection{Compactified configuration space operads}\label{sec:compactified}
Let us also briefly recall the definition of the compactified configuration spaces models $\SFM_n$ for the $n$-Swiss Cheese operads. First consider the space $\Conf_r(\R^n)$ of $r$ distinguishable points in $\R^n$. It comes equipped with a natural action of the group $\R_{>0}\ltimes \R^n$ by scaling and translation. The Fulton-MacPherson-Axelrod-Singer compactification of the quotient
$$
\FM_n(r) = (\Conf_r(\R^n)/\R_+\ltimes \R^n)^-
$$
is defined by iterated real blowups (or rather bordifications) of the boundary, cf. \cite{GJ, K1, K2}. The spaces $\FM_n(-)$ naturally assemble into a topological operad $\FM_n$, with the operadic conposition defined by ``inserting'' one configuration of points at some point of another. The composition takes values on the boundary.
One may or may not define $\FM_n(0)$ to consist of a single point, insertion of which is the forgetful map. To distinguish both variants we will call the ``unital version'' $\uFM_n$, with $\uFM_n(0)=*$, while $\FM_n(0)=\emptyset$.
In this note, we will mostly disconsider this zero-ary operation until Remark \ref{rem:zeroary}.

Similarly, we consider the the space of configurations $\Conf_{r,s}(\R_{\geq 0}\times\R^{n-1})$ of $r$ points in the upper halfspace $\R_{> 0}\times\R^{n-1}$ and $s$ points on the boundary $\R^{n-1}$. This configuration space admits a natural action of the group $\R_{>0}\ltimes \R^{n-1}$ by scaling and translation parallel to the boundary.
One defines again the Fulton-MacPherson-Axelrod-Singer compactification of the quotient
\[
 \FMH_n(r,s) = (\Conf_{r,s}(\R_{\geq 0}\times\R^{n-1})/\R_+\ltimes \R^{n-1})^-.
\]
Again, by insertion of configurations the pair $\SFM_n=(\FM_n,\FMH_n)$ defines a topological operad \cite{voronov_swisscheese}, which we will use as our variant of the $n$-Swiss Cheese operad.
Finally, we will consider another a variant $\uSFM_n=(\uFM_n,\uFMH_n)$ by declaring that $\uFMH_n(0,0)=*$, the additional operation acting by forgetting points.

\section{Graphical (co)operads}
We briefly recall here the definition of several combinatrial (co)operads. Most definitions have appeared in the literature, see, e.g., \cite{LV, graphthings, K2, K1}.
Let $\gra_{r,k}$ be the set of directed graphs with vertex set $\{1,\dots,r\}$ and edges labelled by $\{1,\dots,k\}$. It carries a natural action of the group $S_k\ltimes (S_2)^k$ by permuting the edge labels and edge directions.
One defines vector spaces 
\[
\dGra_n(r) := \bigoplus_k\left( \R\langle \gra_{r,k} \rangle\otimes (\R[n-1])^{\otimes k}\right)_{S_k}
\]
where the action of $S_k$ is the diagonal one, by permutation of edge labels and of the weight factors $\K[n-1]$, with appropriate Koszul signs.
The spaces $\dGra_n(r)$ assemble into a cooperad $\dGra_n$, with the right $S_r$ action by permuting vertex labels, and the cocomposition by "subgraph contraction", cf. Figure \ref{fig:dGracocomp}. 
Futhermore, $\dGra_n$ is a Hopf cooperad, with the Hopf structure defined by gluing two graphs at their vertices.

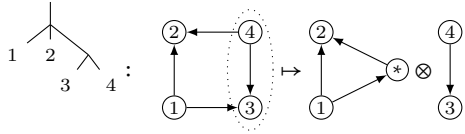
\begin{figure}
 \[
   \begin{tikzpicture}[baseline=-.65ex]
  \node (v1) at (0,.2) {$\scriptstyle 1$};
  \node (v2) at (0.5,.2) {$\scriptstyle 2$};
  \coordinate (vp) at (1,.2) ;
  \coordinate (vpp) at (.5,.6) ;
  \node (v3) at (0.7,-.2) {$\scriptstyle 3$};
  \node (v4) at (1.3,-.2) {$\scriptstyle 4$};
  \draw (vpp) edge (v1) edge (v2) edge (vp) edge +(0,.3)
  (vp) edge (v3) edge (v4);
 \end{tikzpicture}
 \colon\quad
  \begin{tikzpicture}[baseline=-.65ex]
   \node[ext] (v1) at (-.5,-.5) {$\scriptstyle 1$};
   \node[ext] (v2) at (.5,-.5) {$\scriptstyle 3$};
   \node[ext] (v3) at (.5,.5) {$\scriptstyle 4$};
   \node[ext] (v4) at (-.5,.5) {$\scriptstyle 2$};
   \draw[-latex] (v1) edge (v4) edge (v2)
   (v3) edge (v4) edge (v2);
   \draw[dotted] (.5,0) ellipse (.3 and .8); 
  \end{tikzpicture}
\mapsto 
  \begin{tikzpicture}[baseline=-.65ex]
   \node[ext] (v1) at (-.5,-.5) {$\scriptstyle 1$};
   \node[ext] (vx) at (.5,0) {$\scriptstyle *$};
   \node[ext] (v4) at (-.5,.5) {$\scriptstyle 2$};
   \draw[-latex] (v1) edge (v4) edge (vx)
   (vx) edge (v4);
  \end{tikzpicture}
  \otimes 
    \begin{tikzpicture}[baseline=-.65ex]
   \node[ext] (v2) at (.5,-.5) {$\scriptstyle 3$};
   \node[ext] (v3) at (.5,.5) {$\scriptstyle 4$};
   \draw[-latex] (v3) edge (v2);
  \end{tikzpicture}
 \]
 \caption{\label{fig:dGracocomp} A example of a cooperadic cocomposition in the cooperad $\dGra_n$, corresponding to the tree shown on the left.
 }
\end{figure}

There is furthermore an interesting quotient Hopf cooperad
\[
\Gra_n \leftarrow \dGra_n
\]
obtained by identifying a directed edge with $(-1)^{n}$ times the oppositely directed edge.
The Hopf operad $\Gra_n$ (and hence also $\dGra_n$) comes equipped with a natural map 
\[
\Gra_n \to \Omega_\PA(\FM_n)
\]
by sending a graph $\Gamma$ to the differentlal form
\[
\omega_\Gamma' := \bigwedge_{(i,j)} \pi_{ij}^* \Omega_{S^{n-1}}
\]
where the product is over the edges in the graph, 
\[
\pi_{ij} : \FM_n(r) \to \FM_n(2) \cong S^{n-1}
\]
is the forgetful map forgetting all points but points $i$ and $j$, and finally $\Omega_{S^{n-1}}$ is the round volume form on the sphere $S^{n-1}$.
Furthermore, there is a natural map of Hopf cooperads $\Gra_n\to \Poiss_n^*$, (pre-)dual to the map of Hopf operads $\Poiss_n\to \Gra_n^*$ which maps the product and bracket generators as follows:
\begin{align*}
- \wedge - &\mapsto 
\begin{tikzpicture}[baseline=-.65ex]
\node[ext] (v) at (0,0) {1};
\node[ext] (w) at (0.7,0) {2};
\end{tikzpicture}
&
[-,-] &\mapsto 
\begin{tikzpicture}[baseline=-.65ex]
\node[ext] (v) at (0,0) {1};
\node[ext] (w) at (1,0) {2};
\draw (v) edge (w) ;
\end{tikzpicture}\,.
\end{align*}

From the cooperads $\Gra_n$ and $\dGra_n$ we may build further cooperads by (co)operadic twisting \cite{vastwisting}. 
Concretely, we define collections $\fGraphs_n$, $\fdGraphs_n$ such that 
\begin{align*}
\fGraphs_n(r) &= \bigoplus_{r'\geq 0} (\Gra_n(r+r')\otimes (\K[-n])^{\otimes r'} )_{S_{r'}} \\
\fdGraphs_n(r) &= \bigoplus_{r'\geq 0} (\dGra_n(r+r')\otimes (\K[-n])^{\otimes r'} )_{S_{r'}}.
\end{align*}
In other words, one formally "fills" some of the "slots" in cooperations. In pictures, we will draw the corresponding (unlabeled) vertices black, cf. \eqref{equ:graphsexample}.
We furthermore call the $r$ labelled vertices the "external" vertices and the ($r'$) black vertices the "internal" vertices.
The collections $\fGraphs_n$ and $\fdGraphs_n$ carry natural dg Hopf cooperad structures.
The differential is defined by contracting edges, removing one internal vertex, e.g., 
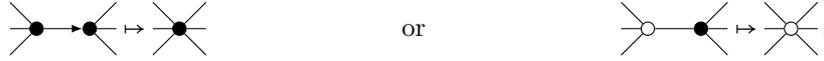
\begin{align*}
 \begin{tikzpicture}[baseline=-.65ex, scale=.7]
  \node[int] (v) at (0,0) {};
  \node[int] (w) at (1,0) {};
  \draw[-latex] (v) edge (w);
  \draw (v) edge +(-.5,.5) edge +(-.5,0) edge +(-.5,-.5);
  \draw (w) edge +(.5,.5) edge +(.5,0) edge +(.5,-.5);
 \end{tikzpicture}
&\mapsto
 \begin{tikzpicture}[baseline=-.65ex, scale=.7]
  \node[int] (v) at (0,0) {};
  \draw (v) edge +(-.5,.5) edge +(-.5,0) edge +(-.5,-.5);
  \draw (v) edge +(.5,.5) edge +(.5,0) edge +(.5,-.5);
 \end{tikzpicture}
 &\text{or}
 & &
  \begin{tikzpicture}[baseline=-.65ex, scale=.7]
  \node[ext] (v) at (0,0) {};
  \node[int] (w) at (1,0) {};
  \draw (v) edge (w);
  \draw (v) edge +(-.5,.5) edge +(-.5,0) edge +(-.5,-.5);
  \draw (w) edge +(.5,.5) edge +(.5,0) edge +(.5,-.5);
 \end{tikzpicture}
&\mapsto
 \begin{tikzpicture}[baseline=-.65ex, scale=.7]
  \node[ext] (v) at (0,0) {};
  \draw (v) edge +(-.5,.5) edge +(-.5,0) edge +(-.5,-.5);
  \draw (v) edge +(.5,.5) edge +(.5,0) edge +(.5,-.5);
 \end{tikzpicture}\, .
\end{align*}
The Hopf structure is defined as before by gluing graphs at their external vertices, and the finally the cooperad structure is also defined as before by subgraph contraction.

There are a natural maps of Hopf cooperads $\Gra_n\to \fGraphs_n$, and the map $\Gra_n\to \Omega(\FM_n)$ from before factors through $\fGraphs_n$. More concretely, a graph $\Gamma\in \Gra_n(r+r')$ is sent to the differential form 
\begin{equation}\label{equ:Graphsintegral}
\omega_\Gamma = \int_{\FM_n(r+r')\to \FM_n(r)} \omega_\Gamma'
\end{equation}
where $\int_{\FM_n(r+r')\to \FM_n(r)}: \Omega(\FM_n(r+r'))\to \Omega(r)$ is the fiber integral, for a more precise definition of this operation we refer to \cite{HLTV}.

The map $\fGraphs_n\to \Omega(\FM_n)$ is not a quasi-isomorphism. However, one may pass to a quotient operad
\[
\Graphs_n \leftarrow \fGraphs_n
\]
obtained by equating graphs to zero that contain connected components of only internal vertices. It is easy to check that for such graphs $\Gamma$ we have $\omega_{\Gamma}=0$ so that one obtains a map  $\Graphs_n\to \Omega(\FM_2)$.
Now the main Theorem by Kontsevich, Lambrechts and Voli\'c is that this map is a quasi-isomorphism.
\begin{thm}[Kontsevich \cite{K2}, Lambrechts-Voli\'c \cite{LV}]
The map
\[
\Graphs_n \to \Poiss_n^*
\]
is a quasi-isomorphism of Hopf cooperads. The map 
\[
\Graphs_n\to \Omega_{\PA}(\FM_2)
\]
is a quasi-isomorphism of Hopf collections compatible with the (co)operadic compositions in the sense that the diagram \ref{equ:opcompat1} commutes.
\end{thm}

The goal of the present paper is to describe the generalization of the above construction to the Swiss Cheese setting.
Define $\kgra_{r,s,k}$ to be the set of directed graphs with vertex set $\{1,\dots,r \}\sqcup \{\bar 1,\dots,\bar s\}$, and directed edges numbered $\{1,\dots, k\}$, such that no edge starts at a vertex in $\{\bar 1,\dots,\bar s\}$.
Following Kontsevich, we call the vertices $\{1,\dots,r \}$ the type I vertices and the $\{\bar 1,\dots,\bar s\}$ the type II vertices.
We define collections 
\[
\KGra_n(r,s) = 
\begin{cases}
\bigoplus_k\left( \R\langle \kgra_{r,s,k} \rangle\otimes (\R[n-1])^{\otimes k}\right)_{S_k} &\text{for $n>1$} \\
\bigoplus_k\left( \R\langle \kgra_{r,s,k} \rangle\otimes (\R[n-1])^{\otimes k} \otimes \R[S_s] \right)_{S_k} &\text{for n=1}
\end{cases}.
\]
Note that we distinguish the case $n=1$, and the factor $\R[S_s]$ occurring in the definition may be interpreted as equipping our graphs with additional data, namely an additional linear ordering on the set of type II vertices.
The bi-collection $\KGra_n$ carries a natural right comodule structure over the operad $\dGra_n$ by contracting subgraphs of type I vertices only. Furthermore, $\dGra_n$ and $\KGra_n$ assemble into a two-colored operad $\SGra_n$ such that 
\begin{align*}
\SGra_n^1(r,s) &=  
\begin{cases}
\dGra_n(r) & \text{for $s=0$} \\
0 & \text{otherwise}
\end{cases} 
\\
\SGra_n^2(r,s) &= \KGra_n(r,s).
\end{align*}

The map $\dGra_n\to \Omega(\FM_n)$ may be extended to a map of colored collections 
\begin{equation}\label{equ:SGraToSFM}
\SGra_n \to \Omega(\SFM_n)
\end{equation}
by sending a graph $\Gamma$ in $\KGra_n(r,s)$ to the differential form 
\[
\omega_\Gamma' = \bigwedge_{(i,j)\text{ edge}} \pi_{ij}^* \tilde \Omega
\]
where $\pi_{ij}$ is the forgetful map forgetting all but vertices $i$ and $j$ and $\tilde \Omega$ is the hyperbolic $n-1$-form described in the introduction, cf. \eqref{equ:hyperbolicanglepicture}.

The map \eqref{equ:SGraToSFM} is not a quasi-isomorphism (of course). Again, to repair the defect we may pass to a twisted version, using the colored (co)operadic twisting construction, cf. \cite{graphthings}, \cite{ricardoBr}. 
Concretely, this means that we extend the cooperad $\fdGraphs_n$ to a two colored cooperad $\fSGraphs_n=(\fGraphs_n,\fKGraphs_n)$, whose color 2 components are given by the bi-collection
\[
\fKGraphs_n(r,s) = \bigoplus_{r',s'} \left( \KGra_n(r+r',s+s') \otimes (\R[-n])^{\otimes r'} \otimes (\R[1-n])^{\otimes s'} \right)_{S_{r'}\times S_{s'}}.
\]

Elements may be depicted as graphs with four kinds of vertices: internal and external vertices, of type either I or II, cf. \eqref{equ:kgraphsexample}.
Again, similarly to \eqref{equ:Graphsintegral} we define for such a graph $\Gamma$ in $\KGra_n(r+r',s+s')$ a PA form $\omega_\Gamma\in \Omega_\PA(\FMH(r,s))$ by a fiber integral 
\[
\omega_\Gamma = \int_{\FMH_n(r+r',s+s')\to \FMH_n(r,s)} \omega_\Gamma'
\]
along the fiber of the forgetful map $\FMH(r+r', s+s')\to \FMH(r,s)$. Again, we refer to \cite{HLTV} for the more precise definition of the fiber integral for PA forms.

We may equip $\fSGraphs_n$ with a two colored dg Hopf operad structure. The Hopf and cooperad structures are induced from those on $\SGra_n$ and are given by gluing graphs at the external vertices, and by subgraph contraction respectively, as before.
The definition of the differential in the operadic twisting construction is more tricky. It is defined so that the map $\fSGraphs_n\to \Omega_\PA(\SFM_n)$ commutes with the differentials.  
More concretely, the arity $(0,0)$ components of the map $\fSGraphs_n\to \Omega_\PA(\SFM_n)$ assign to each graph $\Gamma$ in $\KGraphs_n(0,0)$ without external vertices, and with say $r'$ and $s'$ internal vertices of types I and II a number
\begin{equation}\label{equ:KontsevichWeight}
c_\Gamma = \int_{\FMH_n(r',s')} \omega_\Gamma'.
\end{equation}
The differential on a general graph $\Gamma\in \KGraphs_n(r,s)$ now has 3 pieces: First, we may contract an edge connecting to least one internal type I vertex.
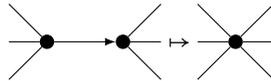
\[
\begin{tikzpicture}[baseline=-.65ex]
\node[int] (v) at (0,0) {};
\node[int] (w) at (1,0) {};
\draw (v) edge +(-.5,-.5) edge +(-.5,0) edge +(-.5, .5);
\draw (w) edge +(.5,-.5) edge +(.5,0) edge +(.5, .5);
\draw[-latex] (v) edge (w);
\end{tikzpicture}
\mapsto
\begin{tikzpicture}[baseline=-.65ex]
\node[int] (v) at (0,0) {};
\draw (v) edge +(-.5,-.5) edge +(-.5,0) edge +(-.5, .5);
\draw (v) edge +(.5,-.5) edge +(.5,0) edge +(.5, .5);
\end{tikzpicture}
\]
Secondly, we may contract a subgraph $\Gamma_1$ consisting of at most one external vertex, with coefficient $c_{\Gamma_1}$.
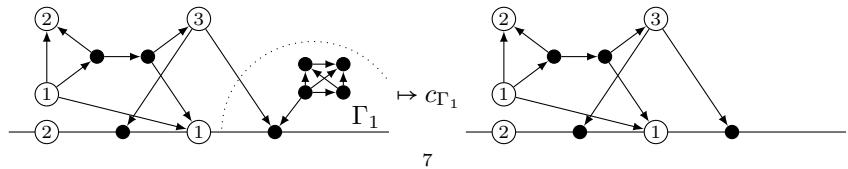
\[
  \begin{tikzpicture}[baseline=-.65ex]
  \draw (-.5,-.5) -- (4.5,-.5);
  \node[ext] (v1) at (0,0) {$\scriptstyle 1$};
  \node[ext, fill=white] (v2) at (2,-.5) {$\scriptstyle 1$};
  \node[ext] (v3) at (0,1) {$\scriptstyle 2$};
  \node[ext] (v4) at (2,1) {$\scriptstyle 3$};
  \node[ext, fill=white] (v5) at (0,-.5) {$\scriptstyle 2$};
  \node[int] (w1) at (.66,.5) {};
  \node[int] (w2) at (1.33,.5) {};
  \begin{scope}[xshift=.4cm, yshift=-.1cm]
  \node[int] (w3) at (3,.125) {};
  \node[int] (w4) at (3.5,.125) {};
  \node[int] (w5) at (3,.5) {};
  \node[int] (w6) at (3.5,.5) {};
  \end{scope}
  \node[int] (u1) at (1,-.5) {};
  \node[int] (u2) at (3,-.5) {};
  \draw[-latex] (v1) edge (v2) edge (v3) edge (w1)
        (w1) edge (w2) edge (v3)
        (w2) edge (v2) edge (v4)
        (w3) edge (w4) edge (w5) edge (w6)
        (w4) edge (w5) edge (w6)
        (w5) edge (w6)
        (v4) edge (u1) edge(u2)
        (w3) edge (u2);       
 \node at (4.2,-.3) {$\Gamma_1$};
 \clip (2.2,-0.5) rectangle (4.5,2);
  \draw[dotted] (3.5,-.5) circle (1.2);
 \end{tikzpicture}
 \mapsto 
 c_{\Gamma_1} \, 
   \begin{tikzpicture}[baseline=-.65ex]
  \draw (-.5,-.5) -- (4.5,-.5);
  \node[ext] (v1) at (0,0) {$\scriptstyle 1$};
  \node[ext, fill=white] (v2) at (2,-.5) {$\scriptstyle 1$};
  \node[ext] (v3) at (0,1) {$\scriptstyle 2$};
  \node[ext] (v4) at (2,1) {$\scriptstyle 3$};
  \node[ext, fill=white] (v5) at (0,-.5) {$\scriptstyle 2$};
  \node[int] (w1) at (.66,.5) {};
  \node[int] (w2) at (1.33,.5) {};
  \begin{scope}[xshift=.4cm, yshift=-.1cm]
  \end{scope}
  \node[int] (u1) at (1,-.5) {};
  \node[int] (u2) at (3,-.5) {};
  \draw[-latex] (v1) edge (v2) edge (v3) edge (w1)
        (w1) edge (w2) edge (v3)
        (w2) edge (v2) edge (v4)
        (v4) edge (u1) edge (u2);       
 \end{tikzpicture}
\]
Finally, we identify a subgraph $\Gamma'$ consisting of all external vertices. The third piece of the differential then sends $\Gamma$ to $\pm c_{\Gamma/\Gamma'} \Gamma'$.
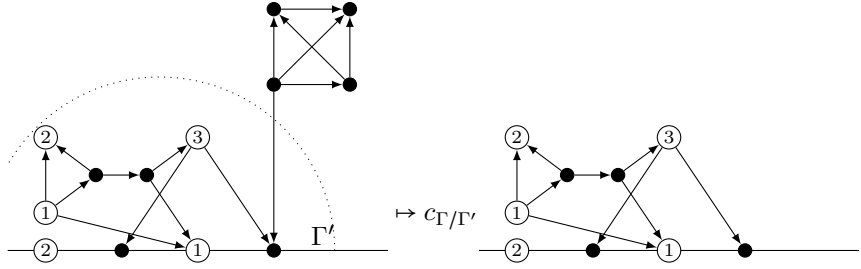
\[
  \begin{tikzpicture}[baseline=-.65ex]
  \draw (-.5,-.5) -- (4.5,-.5);
  \node[ext] (v1) at (0,0) {$\scriptstyle 1$};
  \node[ext, fill=white] (v2) at (2,-.5) {$\scriptstyle 1$};
  \node[ext] (v3) at (0,1) {$\scriptstyle 2$};
  \node[ext] (v4) at (2,1) {$\scriptstyle 3$};
  \node[ext, fill=white] (v5) at (0,-.5) {$\scriptstyle 2$};
  \node[int] (w1) at (.66,.5) {};
  \node[int] (w2) at (1.33,.5) {};
    \begin{scope}[xshift=0cm, yshift=1.2cm]
  \node[int] (w3) at (3,.5) {};
  \node[int] (w4) at (4,.5) {};
  \node[int] (w5) at (3,1.5) {};
  \node[int] (w6) at (4,1.5) {};
  \end{scope}
  \node[int] (u1) at (1,-.5) {};
  \node[int] (u2) at (3,-.5) {};
  \draw[-latex] (v1) edge (v2) edge (v3) edge (w1)
        (w1) edge (w2) edge (v3)
        (w2) edge (v2) edge (v4)
        (w3) edge (w4) edge (w5) edge (w6)
        (w4) edge (w5) edge (w6)
        (w5) edge (w6)
        (v4) edge (u1) edge(u2)
        (w3) edge (u2);
          \node at (3.65,-.3) {$\Gamma'$};
 \clip (-.5,-0.5) rectangle (4.5,4.5);
  \draw[dotted] (1.5,-.5) circle (2.3);       
 \end{tikzpicture}
 \mapsto 
 c_{\Gamma/\Gamma'}
    \begin{tikzpicture}[baseline=-.65ex]
  \draw (-.5,-.5) -- (4.5,-.5);
  \node[ext] (v1) at (0,0) {$\scriptstyle 1$};
  \node[ext, fill=white] (v2) at (2,-.5) {$\scriptstyle 1$};
  \node[ext] (v3) at (0,1) {$\scriptstyle 2$};
  \node[ext] (v4) at (2,1) {$\scriptstyle 3$};
  \node[ext, fill=white] (v5) at (0,-.5) {$\scriptstyle 2$};
  \node[int] (w1) at (.66,.5) {};
  \node[int] (w2) at (1.33,.5) {};
  \begin{scope}[xshift=.4cm, yshift=-.1cm]
  \end{scope}
  \node[int] (u1) at (1,-.5) {};
  \node[int] (u2) at (3,-.5) {};
  \draw[-latex] (v1) edge (v2) edge (v3) edge (w1)
        (w1) edge (w2) edge (v3)
        (w2) edge (v2) edge (v4)
        (v4) edge (u1) edge (u2);       
 \end{tikzpicture}
\]

We emphasize that the differential on $\KGraphs_n$ is relatively complicated, involving many non-zero terms. In particular, for $n=2$ the cooperad $\KGraphs_2$ encodes a stable formality quasi-isomorphism in the sense of \cite{vasilystable}.

As before the map 
\[
\fKGraphs_n\to \Omega(\FMH_n)
\]
is not (yet) a quasi-isomorphism. We shall repair this defect in the following section.


\section{The (bi-)ideal $\fSGraphs_n^{disc}$}\label{sec:ideal}
We split the construction of the bi-ideal $\fSGraphs_n^{disc}$ in two cases.

For $n\geq 3$ we define the bi-ideal $\fSGraphs_n^{disc}$ in the obvious way by defining a graph $\Gamma\in \fKGraphs_n$ to be externally disconnected if it contains a connected component made of internal vertices only.
We note that this definition yields a bi-ideal, i.e., it is compatible with the dg Hopf cooperad structure by the following reasoning. We only consider the ``novel'' piece of that statement, i.e., the bi-ideal conditions on $\fGraphs_n^{1,disc}\subset \fGraphs_n^{1,disc}$, and take the statements on $\fGraphs_n^{disc}$ for granted.
\begin{itemize}
 \item By inspection it is clear that the cooperadic cocomposition cannot co-compose an externally disconnected graph into two externally connected graphs. Hence $\fSGraphs_n^{disc}$ is a cooperadic coideal.
 \item The commutative product (fusion at external vertices) cannot remove connected components of internal vertices only, so each $\fSGraphs_n^{1,disc}(r,s)$ is a graded ideal in the graded algebra $\fSGraphs_n^1(r,s)$.
\item The subset $\fSGraphs_n^{disc}$ is closed under the differential since the numbers $c_\Gamma$ of \eqref{equ:KontsevichWeight} defined above are zero for disconnected graphs $\Gamma$ and $n\geq 3$.
\end{itemize}
Hence we may define the two-colored quotient dg Hopf cooperad 
\[
\SGraphs_n = \fSGraphs_n / \fSGraphs_n^{disc}.
\]

To check that the map $\fSGraphs_n\to \Omega_\PA(\SFM_n)$ factors through the quotient $\SGraphs_n$, we have to check that the fiber integrals \eqref{equ:integral2} vanish on externally disconnected graphs $\Gamma$. However, this is pretty obvious: Suppose some $k'$ type I and $l'$ type II vertices form a connected component without external vertices. Using Fubini's theorem to factor out the integral over the positions of these internal vertices, the fiber integral then reads 
\[
 \omega_\Gamma = \int_{\FM_n^1(r+k-k',s+l-l')\to\FM_n^1(r,s)}
 \int_{\FM_n^1(r+k,s+l)\to\FM_n^1(r+k-k',s+l-l')}
 \bigwedge_{ (i,j)\text{ edge} }
 \pi_{ij}^*\Omega_{S^{n-1}} .
\]
The inner integral yields zero, because the piece of the pieces of the integrand depending on the locations of the $k'+l'$ internal points are basic under scaling and translation and can thus not contribute a top degree piece to the differential form.

\begin{rem}
Note that instead of defining $\SGraphs_n$ as a quotient as above, we could have defined it to consist only of externally connected graphs from the start. This would have avoided or made obvious some parts of the above discussion. However, we next want to treat the case $n=2$, where the above statements are not immediate.
\end{rem}

The case $n=2$: 
First note that in the case $n=2$ one cannot define $\fSGraphs_2^{disc}$ in the same way as before, since there is the disconnected graph 
$$
\Gamma = 
\begin{tikzpicture}[baseline=-.65ex]
\draw(0,0) -- (1,0);
\node[int] at (.2,0) {};
\node[int] at (.8,0) {};
\end{tikzpicture}
$$ 
which is assigned nonzero weight $c_\Gamma=1\neq 0$ under \eqref{equ:KontsevichWeight}.
Hence the naively defined subspace of externally disconnected graphs is not closed under the differential, i.e., the differential can map a disconnected graph to a connected one.
To find a better definition of ``externally disconnected'', we will first consider a change of basis in $\fSGraphs_2^{disc}$. By a Lie decorated graph we mean a graph $\Gamma$ as in $\fSGraphs_2$, but with a linear ordering on the \emph{external} type II vertices only, 
and with an arrangement of the internal type II vertices in a product of formal Lie expressions. It is best explained by a picture:
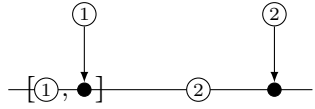
\[
 \begin{tikzpicture}
 \draw (-.5,0) -- (3.5,0);
  \node[ext] (v1) at (0.5,1) {$\scriptstyle 1$};
  \node[ext] (v2) at (3,1) {$\scriptstyle 2$};
  \node[ext, fill=white](u1) at (0,0) {$\scriptstyle 1$};
  \node[ext, fill=white](u2) at (2,0) {$\scriptstyle 2$};
  \node at (-.2, 0) {$ [$};
  \node at (.7, 0) {$ ]$};
  \node at (.25, -.1) {$,$};
  \node[int] (w1) at (.5,0) {};
  \node[int] (w2) at (3,0) {};
  \draw [-latex] (v1) edge (w1) (v2) edge (w2);
 \end{tikzpicture}
\]
Such Lie decorated graphs naturally define elements of $\fSGraphs_2$ by formally applying the Poincar\'e-Birkhoff-Witt map. Again, this procedure is best explained by an example. The Lie decorated graph above corresponds to the linear combination of ordinary graphs
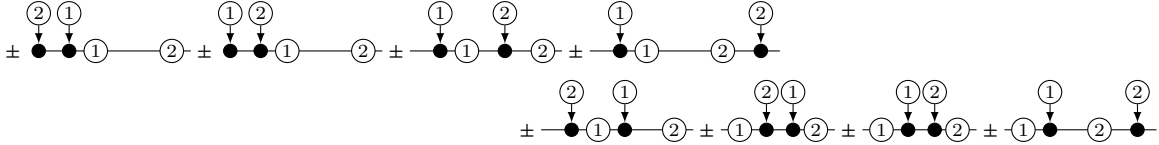
\begin{multline*}
 \pm
 \begin{tikzpicture}[baseline=-.65ex, scale=.5]
 \draw (-1.5,0) -- (2.5,0);
  \node[ext] (v1) at (-0.7,1) {$\scriptstyle 1$};
  \node[ext] (v2) at (-1.5,1) {$\scriptstyle 2$};
  \node[ext, fill=white](u1) at (0,0) {$\scriptstyle 1$};
  \node[ext, fill=white](u2) at (2,0) {$\scriptstyle 2$};
  \node[int] (w1) at (-.7,0) {};
  \node[int] (w2) at (-1.5,0) {};
  \draw [-latex] (v1) edge (w1) (v2) edge (w2);
 \end{tikzpicture}
\pm
  \begin{tikzpicture}[baseline=-.65ex, scale=.5]
 \draw (-1.5,0) -- (2.5,0);
  \node[ext] (v1) at (-1.5,1) {$\scriptstyle 1$};
  \node[ext] (v2) at (-.7,1) {$\scriptstyle 2$};
  \node[ext, fill=white](u1) at (0,0) {$\scriptstyle 1$};
  \node[ext, fill=white](u2) at (2,0) {$\scriptstyle 2$};
  \node[int] (w1) at (-1.5,0) {};
  \node[int] (w2) at (-.7,0) {};
  \draw [-latex] (v1) edge (w1) (v2) edge (w2);
 \end{tikzpicture}
\pm
  \begin{tikzpicture}[baseline=-.65ex, scale=.5]
 \draw (-1.5,0) -- (2.5,0);
  \node[ext] (v1) at (-0.7,1) {$\scriptstyle 1$};
  \node[ext] (v2) at (1,1) {$\scriptstyle 2$};
  \node[ext, fill=white](u1) at (0,0) {$\scriptstyle 1$};
  \node[ext, fill=white](u2) at (2,0) {$\scriptstyle 2$};
  \node[int] (w1) at (-.7,0) {};
  \node[int] (w2) at (1,0) {};
  \draw [-latex] (v1) edge (w1) (v2) edge (w2);
 \end{tikzpicture}
 \pm
   \begin{tikzpicture}[baseline=-.65ex, scale=.5]
 \draw (-1.5,0) -- (3.5,0);
  \node[ext] (v1) at (-0.7,1) {$\scriptstyle 1$};
  \node[ext] (v2) at (3,1) {$\scriptstyle 2$};
  \node[ext, fill=white](u1) at (0,0) {$\scriptstyle 1$};
  \node[ext, fill=white](u2) at (2,0) {$\scriptstyle 2$};
  \node[int] (w1) at (-.7,0) {};
  \node[int] (w2) at (3,0) {};
  \draw [-latex] (v1) edge (w1) (v2) edge (w2);
 \end{tikzpicture}
 \\
 \pm
 \begin{tikzpicture}[baseline=-.65ex, scale=.5]
 \draw (-1.5,0) -- (2.5,0);
  \node[ext] (v1) at (0.7,1) {$\scriptstyle 1$};
  \node[ext] (v2) at (-.7,1) {$\scriptstyle 2$};
  \node[ext, fill=white](u1) at (0,0) {$\scriptstyle 1$};
  \node[ext, fill=white](u2) at (2,0) {$\scriptstyle 2$};
  \node[int] (w1) at (.7,0) {};
  \node[int] (w2) at (-.7,0) {};
  \draw [-latex] (v1) edge (w1) (v2) edge (w2);
 \end{tikzpicture}
 \pm
  \begin{tikzpicture}[baseline=-.65ex, scale=.5]
 \draw (-.5,0) -- (2.5,0);
  \node[ext] (v1) at (1.4,1) {$\scriptstyle 1$};
  \node[ext] (v2) at (.7,1) {$\scriptstyle 2$};
  \node[ext, fill=white](u1) at (0,0) {$\scriptstyle 1$};
  \node[ext, fill=white](u2) at (2,0) {$\scriptstyle 2$};
  \node[int] (w1) at (1.4,0) {};
  \node[int] (w2) at (.7,0) {};
  \draw [-latex] (v1) edge (w1) (v2) edge (w2);
 \end{tikzpicture}
  \pm
  \begin{tikzpicture}[baseline=-.65ex, scale=.5]
 \draw (-.5,0) -- (2.5,0);
  \node[ext] (v1) at (.7,1) {$\scriptstyle 1$};
  \node[ext] (v2) at (1.4,1) {$\scriptstyle 2$};
  \node[ext, fill=white](u1) at (0,0) {$\scriptstyle 1$};
  \node[ext, fill=white](u2) at (2,0) {$\scriptstyle 2$};
  \node[int] (w1) at (.7,0) {};
  \node[int] (w2) at (1.4,0) {};
  \draw [-latex] (v1) edge (w1) (v2) edge (w2);
 \end{tikzpicture}
 \pm
 \begin{tikzpicture}[baseline=-.65ex, scale=.5]
 \draw (-.5,0) -- (3.5,0);
  \node[ext] (v1) at (0.7,1) {$\scriptstyle 1$};
  \node[ext] (v2) at (3,1) {$\scriptstyle 2$};
  \node[ext, fill=white](u1) at (0,0) {$\scriptstyle 1$};
  \node[ext, fill=white](u2) at (2,0) {$\scriptstyle 2$};
  \node[int] (w1) at (.7,0) {};
  \node[int] (w2) at (3,0) {};
  \draw [-latex] (v1) edge (w1) (v2) edge (w2);
 \end{tikzpicture} 
\end{multline*}
(To obtain the signs, one needs to consider the type II vertices as odd objects.)

It follows from the Poincar\'e-Birkhoff-Witt Theorem that the collection of (isomorphism classes of) such Lie decorated graphs in fact forms a basis $\fSGraphs_2$. We are now ready to state our definition of $\fSGraphs_2^{disc}$.
\begin{defi}
 A Lie decorated graph $\Gamma\in \fGraphs_2^{1}(r,s)$ is called externally disconnected if it contains a connected component consisting of internal vertices only, where we count type II vertices in a formal Lie expression as belonging to the same connected component. The subspace  $\fGraphs_2^{1,disc}(r,s)\subset \fGraphs_2^{1}(r,s)$ is defined to be the linear span of the externally disconnected Lie decorated graphs.
\end{defi}

We claim that this definition indeed satisfies our requirements.
\begin{prop}
The sub-collection $\fSGraphs_2^{disc}$ made of $\fGraphs_2^{disc}$ in the first color and the spaces $\fGraphs_2^{1,disc}(r,s)$ in the second is a dg Hopf cooperadic bi-ideal.
Furthermore, the map $\fSGraphs_2\to \Omega_{\PA}(\SFM_2)$ factors through the quotient dg Hopf cooperad
\[
 \SGraphs_2=(\Graphs_n,\KGraphs_n) =\fSGraphs_2^{}/\fSGraphs_2^{disc}
\]
\end{prop}
\begin{proof}
 Again we have to check compatibility with the algebra structure, the cooperad structure, and the differential.
 Let us leave the former two points to the reader. For the last step, it suffices to note that under \eqref{equ:KontsevichWeight} all externally disconnected graphs are assigned zero weight: The integral \eqref{equ:integral2} factors again using Fubini's theorem, and by basic-ness under rescaling and translation one factor of the integrand cannot contribute  a top form. A similar argument shows the factorization statement.
\end{proof}

\section{The cohomology of $\fSGraphs_n$}\label{sec:cohom}
In this section we will show that the map 
\[
 \SGraphs_n\to \Omega_{\PA}(\SFM_n)
\]
is a quasi-isomorphism, by computing the cohomology of the left-hand side.
Since the part of the result in the first color is already known (by Theorem \ref{thm:KLV}), we focus on the piece of color 2, i.e., on the maps
\[
 \Graphs_n^1(r,s) \to \Omega_{\PA}(\FM_n^1(r,s)).
\]
We have to consider the cases $n=2$ and $n\geq 3$ separately.
To compute the cohomology of $\Graphs_2^1(r,s)$ we consider a filtration on the number of internal type I vertices. No piece of the differential can create such vertices.
Furthermore, there is only one piece of the differential not destroying internal type I vertices, and that piece is a ``Hochschild type'' differential contracting internal type II vertices, schematically:
\[
 \delta_1\colon 
 \begin{tikzpicture}[baseline=-.65ex]
\draw(0,0) -- (1,0);
\node[int] (v1) at (.2,0) {};
\node[int] (v2) at (.8,0) {};
\draw (v1) edge +(-.3,.3) edge +(0,.3)
      (v2) edge +(.3,.3) edge +(0,.3);
\end{tikzpicture}
\mapsto 
 \begin{tikzpicture}[baseline=-.65ex]
\draw(0,0) -- (1,0);
\node[int] (v1) at (.5,0) {};
\draw (v1) edge +(-.3,.3) edge +(-.1,.3)
      (v1) edge +(.3,.3) edge +(0.1,.3);
\end{tikzpicture}\, .
\]
This is the differential we will see on the first page of the spectral sequence. By standard arguments (cf. \cite[Proposition 74]{graphthings}), we find that the homology of the the complex $(\Graphs_n^1(r,s), \delta_1)$ is spanned by Lie decorated graphs $\Gamma$ such that all internal type II vertices are univalent and in their own Lie cluster, and all external type II vertices are zero valent. (In other words, the corresponding classes of plain graphs are found by antisymmetrizing over the positions of the internal type II vertices.)
The differential on the next page removes exactly one internal type I vertex, together with zero or one edges. Again we may borrow the argument of \cite[Proposition 74]{graphthings} to show that the resulting complex is quasi-isomorphic to $\Graphs_n$ (for each fixed linear ordering of the external type II vertices), and hence the cohomology is identified with $s!$ copies of $\e_n^*(r)$, one for each ordering of the external type II vertices. But this space is exactly the cohomology of $\FM_n^1(r,s)$. It remains to show that the spectral sequence abuts at this point. However, this is immediate since all cohomology classes found may be represented by graphs without internal vertices, which are evidently closed under the full differential.

\medskip

Now let us turn to the case $n\geq 3$. 
We consider a filtration by the total number of internal vertices.
Any piece of the differential reduces the number of internal vertices by at least one, and the piece that reduces the number by exactly one is just the contraction of one edge between an internal type I vertex and another vertex.

We will show by an induction on the arity $(r,s)$ that the cohomology can be identified with $\Poiss_n^*(r)\otimes \Graphs_{n-1}(s)$. More concretely, elements of $\Pois_n^*(r)$ may be represented by components of the graph connecting the $r$ external type I vertices, without involvement of internal vertices.
Elements $\gamma\in \Graphs_{n-1}(s)$ be be represented by replacing each edge by a zigzag of the form
\begin{equation}\label{equ:2zigzag}
 \begin{tikzpicture}
  \draw (-.5,0) -- (1.5, 0);
  \node[ext, fill=white] (v1) at (0,0) {};
  \node[ext, fill=white] (v2) at (1,0) {};
  \node[int] (w) at (.5,.5) {};
  \draw[-latex] (w) edge (v1) edge (v2);
 \end{tikzpicture}
\end{equation}
The procedure is best explained by an example:
\[
\Graphs_{n-1}(3)
\ni
  \begin{tikzpicture}[baseline=-.65ex]
  \node[ext] (v1) at (0:.5) {$\scriptstyle 1$};
  \node[ext] (v2) at (120:.5) {$\scriptstyle 2$};
  \node[ext] (v3) at (240:.5) {$\scriptstyle 3$};
  \node[int] (w) at (0,0) {};
  \draw (w) edge (v1) edge (v2) edge (v3);
 \end{tikzpicture}
 \mapsto
  \begin{tikzpicture}[baseline=-.65ex]
  \draw (-.5,0) -- (3.5, 0);
  \node[ext, fill=white] (v1) at (0,0) {$\scriptstyle 1$};
  \node[ext, fill=white] (v2) at (1,0) {$\scriptstyle 2$};
  \node[ext, fill=white] (v3) at (3,0) {$\scriptstyle 3$};
  \node[int] (w) at (2,0) {};
  \node[int] (u1) at (1.5,.7) {};
  \node[int] (u2) at (1.5,1.4) {};
  \node[int] (u3) at (2.5,.7) {};
  \draw[-latex] (u1) edge (w) edge (v2);
  \draw[-latex] (u2) edge (w) edge (v1);
  \draw[-latex] (u3) edge (w) edge (v3);
 \end{tikzpicture} 
 \in \Graphs_n^1(0,3)
\]
Note again that the order of type II vertices in the last picture is irrelevant since we are in the case $n\geq 3$. We merely draw the horizontal line to distinguish the type II vertices on it from the type I vertices above it. (One should consider the line as an $(n-1)$-plane.)

Now let us prove that the cohomology may be identified with $\e_n^*(r)\otimes \Graphs_{n-1}(s)$.
First we may filter our complex again by the number of non-(internal-bivalent-type I-vertices). The first differential in the associated spectral sequence then removes one internal bivalent type I vertex by contracting an adjacent edge. Similarly to \cite[Appendix K]{grt} the resulting complex factors into subcomplexes, one for each pair of non-bivalent vertices connected by a string of bivalent vertices. Those subcomplexes are acyclic is the string connects an type I and a type II vertex. 
If the string connects two type I vertices the cohomology of the string-complex is one-dimensional, spanned by an edge connecting the two vertices directly, in either direction. If the string connects two type I vertices the cohomology of the string-complex is again-dimensional, connecting the two endpoints by a length 2 zigzag of the form \eqref{equ:2zigzag}.
From this argument one can see that $(\Graphs_n^1(r,s), \delta_1)$ is quasi-isomorphic to its quotient, say $V_{r,s}$, obtained by:
\begin{enumerate}
 \item Equating graphs to zero in which type II vertices are connected to type I vertices other then those in length 2 zigzags \eqref{equ:2zigzag} connecting to another type 2 vertex.
 \item Equating graphs to zero which contain bivalent internal type I vertices, other than those in the length 2 zigzags \eqref{equ:2zigzag} between two external type II vertices.
 \item Equating graphs obtained by changing directions of edges connecting two type I vertices.
\end{enumerate}
Obviously, we may identify 
\[
 V_{r,s} = \Graphs_n(r) \otimes \Graphs_{n-1}(s),
\]
with the differential $\delta_1$ inducing the differential on $\Graphs_n(r)$. From this (and Theorem \ref{thm:KLV}) it follows that 
\[
 H(\Graphs_n^1(r,s), \delta_1) \cong \Poiss_n^*(r)\otimes \Graphs_{n-1}(s).
\]
The next differential $\delta_2$ in the (outer) spectral sequence reduces the number of internal vertices by two. By inspection, we find that this differential is exactly the differential on the factor $\Graphs_{n-1}(s)$, and hence (using Theorem \ref{thm:KLV} again) we find that 
\[
 H( H(\Graphs_n^1(r,s), \delta_1), \delta_2) \cong \Poiss_n^*(r)\otimes \Poiss_{n-1}^*(s).
\]
But the right hand side is exactly $H(\FM_n^1(r,s))$.
Again, we find that the spectral sequence abuts at this stage since all surviving classes are naturally represented by graph cocycles.

This finishes the proof of Theorem \ref{thm:main}.
\hfill\qed

\begin{rem}
 We remark that slightly smaller (but quasi-isomorphic) variants of the graphical cooperad $\SGraphs_n$ may be defined. First, for the piece in color 1, i.e., $\Graphs_n$, one may additionally postulate that all graphs with internal vertices of valence $<3$ are zero, as is done in \cite{K2}. The resulting quotient graph cooperad is quasi-isomorphic to our version. 
 Furthermore, one can check that the map $\Graphs_n\to \Omega_{\PA}(\FM_n)$ factors through this quotient.
 
 Similarly, we may define a smaller (but quasi-isomorphic) variant of $\KGraphs_n$ by declaring that graphs with internal vertices of valence $<2$ are zero, as well as graphs with internal vertices with in- and out valence exactly one, i.e., 
 $\begin{tikzpicture}[baseline=-.65ex]
   \node[int] (v) at (0,0) {};
   \draw[-latex] (v) edge +(.5,0);
   \draw[latex-] (v) edge +(-.5,0);
  \end{tikzpicture}=0$.
Again, one can check that the map $\KGraphs_n\to \Omega_{\PA}(\FMH_n)$ factors through this quotient.
\end{rem}

\begin{rem}\label{rem:zeroary}
One may add a nullary cooperation in to the the cooperad $\Graphs_n$ whose coaction creates external vertices of valence zero, call the resulting Hopf cooperad $\uGraphs_n$.
The map $\Graphs_n\to \Omega_{\PA}(\FM_n)$ above then naturally extends to a map $\uGraphs_n\to \Omega_{\PA}(\uFM_n)$, cf. section \ref{sec:compactified} for the notation.
 
Similarly, one may add a nullary cooperation to $\KGraphs_n$ whose coaction creates external type II vertices of valence zero. Calling the resulting two colored Hopf cooperad $\uSGraphs_n=(\uGraphs_n, \uKGraphs_n)$ one can naturally extend the map $\Graphs_n\to \Omega_{\PA}(\FM_n)$ above to a map $\uSGraphs_n\to \Omega_{\PA}(\uSFM_n)$.
 
\end{rem}

\section{An application: Formality morphisms}\label{sec:application} 
In this section we describe an application of the aforementioned results in the field of deformation quantization. Concretely, we sketch a proof that a(ny) $L_\infty$ formality morphism can be extended to a homotopy Gerstenhaber formality morphism.

\subsection{The Kontsevich formality Theorem in deformation quantization}
Let $\Tpoly$ be the Gerstenhaber algebra of multivector fields on $\R^d$ and $\Dpoly$ be the braces algebra of multi differential operators.
The Kontsevich formality Theorem \cite{K1} states that these two objects are quasi-isomorphic as $\hoLie_2$ algebras. Concretely, Kontsevich constructs an $L_\infty$ morphism $\mU$ between the Lie algebra $\Tpoly[1]$ and the Lie algebra $\Dpoly[1]$ given by a sum-of-graphs formula
\begin{equation}\label{equ:stableformality}
\begin{gathered}
\mU_n : S^n (\Tpoly[2]) \to \Dpoly[2] \\
(v_1,\dots,v_n) \mapsto \sum_{\Gamma} c_\Gamma D_\Gamma(v_1, \dots, v_n)
\end{gathered}
\end{equation}
where 
\begin{itemize}
\item The sum is over graphs forming a basis of $\KGraphs_2(0,0)$, with $n$ type I vertices.
\item The coefficients $c_\Gamma$ are defined by the integral formula \eqref{equ:KontsevichWeight}.
\item $D_\Gamma(v_1, \dots, v_n)$ is a multidifferential operator naturally associated to the graph $\Gamma$, cf. \cite{K1} for details. (The precise formula for $D_\Gamma$ will not be important for us.)
\end{itemize}

Note that the formality morphism is completely specified by providing the numbers $c_\Gamma$ in the formula above.
Formality morphisms of the form \eqref{equ:stableformality} have been christened \emph{stable formality morphisms} in \cite{vasilystable}, under the mild additional condition that the linear piece of the $L_\infty$ morphism agrees with the Hochschild-Kostant-Rosenberg map, that is, the coefficients of the graphs 
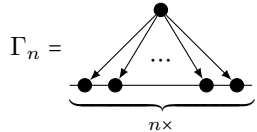
\[
\Gamma_n = 
\underbrace{
\begin{tikzpicture}[baseline=-.65ex]
 \draw(-1.2,-.5) -- (1.2,-.5);
 \node[int] (v) at (0,.5) {};
 \node[int] (w1) at (-1,-.5) {};
 \node[int] (w2) at (-.6,-.5) {};
 \node at (0,-.2) {$\cdots$};
 \node[int] (w3) at (.6,-.5) {};
 \node[int] (w4) at (1,-.5) {};
 \draw[-latex] (v) edge (w1) edge (w2) edge (w3) edge (w4);
\end{tikzpicture}
}_{n\times}
\]
satisfy $c_{\Gamma_n}=\frac 1 {n!}$.

Let us also introduce the following notion.
\begin{defi}
 A stable formality morphism is called \emph{connected} is the numbers $c_\Gamma$ associated to any disconnected\footnote{in the sense discussed above} Lie decorated graph vanish.
\end{defi}

For example, the Kontsevich stable formality morphism is connected.

%
%

\subsection{A model for the operadic bimodule $\lD^1_n(-,0)$}

It has been shown in \cite{graphthings} that the semi-algebraic chains $C(\FMH_2(-,0))$ form an operadic $\Br$-$C(\FM_2)$-bimodule
\[
\Br \aol C(\FMH_2(-,0)) \aor C(\FM_2).
\]
In fact, this operadic bimodule is even an operadic torsor in the sense of \cite{CWtorsors}.
The construction of the operadic bimodule structure in \cite{graphthings} uses only naturally defined operations on the Swiss-Cheese operad $\SFM_2$. Hence, by a similar construction, one can see that our model 
\[
\KGraphsM := \KGraphs_2(-,0)^*
\]
for $C(\FMH_2(-,0))$ also carries a left braces action compatible with the right $\Graphs_2^*$ action, and that all algebraic structures are preserved under the map $C(\FMH_2(-,0))\to \KGraphsM$. In other words, we have a commutative diagram of operads and operadic bimodules (in fact: of operadic torsors)
\[
\begin{tikzcd}[column sep=0.5em]	
\Br \ar{d}{\sim}& \aol & C(\FMH(-,0)) \ar{d}{\sim}& \aor & C(\FM_2)\ar{d}{\sim} \\
\Br & \aol & \KGraphsM & \aor & \Graphs_2^*
\end{tikzcd}
\]

The important point is now that the definition of $\KGraphs_2$ can be repeated, for any connected stable formality morphism $\mU$, which enters the definition in so far that it defines the coefficients \eqref{equ:KontsevichWeight} entering the differential on $\KGraphs_2$. Similar, we may define a variant of the $\Br$-$\Graphs_2^*$ bimodule $\KGraphsM$ for any $\mU$, which we will denote by $\KGraphsM^\mU$.

\subsection{Our result}
Concretely, as an the application of our results we will give a short proof of the following theorem.
\begin{thm}\label{thm:defq}
Suppose $\mU$ is a connected ($L_\infty$-) stable formality morphism. Then one may pick a homotopic stable formality morphism $\mU'$ that can be extended to a stable homotopy Gerstenhaber (or, equivalently homotopy braces) formality morphism. 
\end{thm}

\begin{rem}
 We claim that in fact any stable formality morphism can be changed into a homotopic connected stable formality morphism using \cite{vasilystable}, and hence the above Theorem also holds if we omit the word ``connected''. However, we shall not prove this here.
\end{rem}

The proof is relatively short. 
First, we consider the $\Br$-$\Graphs_2$ bimodule $\KGraphsM^\mU$. It generates a two colored operad 
\[
\bigGraphs := \bpm \Br & \KGraphsM^\mU & \Graphs_2^* \epm
\]
which naturally acts on the colored vector space $\Dpoly\oplus \Tpoly$. 
Let $\ELie$ be the two-colored operad governing two $\hoLie_2$ algebras and an $\infty$-morphism between them, and let similarly $\EGer$ be the two-colored operad governing two $\hoPoiss_2$ algebras and an $\infty$-morphism between them. 
By twisting our stable formality morphism $\mU$ gives a map 
\[
 \ELie \to \bigGraphs
\]
extending the canonical maps $\hoLie_2\to \Br$ and $\hoLie_2\to \Graphs_2^*$.
Our objective is to define a stable homotopy Gerstenhaber formality morphism by extending the above map to a (suitable) map from the operad $\EGer$ into $\bigGraphs$.
This is done by the following arguments:
\begin{enumerate}
\item We know that $H(\KGraphsM^\mU(1))\cong \R$. Pick some cocycle generating the non-trivial cohomology $\bo\in \KGraphsM(1)$.
\item The maps (of dg collections) $\Br \to \KGraphsM^\mU \leftarrow \Graphs_2^*$ obtained by composing with the element $\bo$ are quasi-isomorphisms. Hence we conclude that $\KGraphsM^\mU$ is an operadic $\Br$-$\Graphs_2^*$-torsor in the sense of \cite{CWtorsors}.
\item 
The main result of loc. cit. states that if $\op M$ is an operadic $\op P$-$\op Q$ torsor, then $\op P$ and $\op Q$ are quasi-isomorphic, and furthermore the triple 
\[
\op P \aol \op M \aor \op Q
\] 
is quasi-isomorphic to the "canonical" triple 
\[
\op Q \aol \op Q \aor \op Q.
\] 
It follows further that the two colored operad generated by our triple is quasi-isomorphic to the two colored operad governing two $\op Q$ algebras and a map between them. Furthermore, if we take any cofibrant model of that latter operad, a direct quasi-isomorphism to $\bpm \op P & \op M & \op Q\epm$ may be constructed by lifting.
\item In our setup $\op Q=\Graphs_2^*$ is quasi-isomorphic to the Gerstenhaber operad $\Poiss_2$, cf. Theorem \ref{thm:KLV}. 
The operad $\EGer$ governing to $\hoPoiss_2$ algebras and an $\infty$ morphism between them is a cofibrant replacement for the two colored operad $\bpm \op Q & \op Q & \op Q\epm$. Hence we obtain "for free" from the above generalities a quasi-isomorphism of colored operads
\begin{equation}\label{equ:prestableformality}
\EGer \to \bpm \Br & \KGraphsM^\mU & \Graphs_2^* \epm.
\end{equation}
\end{enumerate}
Having constructed the map \eqref{equ:prestableformality}, we want to verify that it indeed defines a stable formality morphism, possibly after changing to a homotopic morphism.
(I.e., that (i) the induced maps $\hoLie_2 \to \Graphs_2$ and $\hoLie_2\to \Br$ are the standard maps and (ii) the unary component is a deformation of the HKR map.)
The construction of \eqref{equ:prestableformality} is somewhat inexplicit, and we do not desire to unpack the underlying formalism. 
Fortunately, however, the space of homology non-trivial maps $\hoLie_2\to \op P$ for $\op P$ any operad quasi-isomorphic to $\Poiss_2$ is very simple, up to homotopy: There is a one parameter family of homotopy classes of maps, indexed by the rescaling factor applied to the single generator in homology. This scaling freedom is exhausted by the automorphisms of the operad $\op P$ (or any operad) by resealing by arity, i.e., multiplying an arity $r$ operation $x$ by $\lambda^{r-1}$ for $\lambda\in \R_{\neq 0}$.
Hence, by changing the morphism \eqref{equ:prestableformality} to a homotopic one, and by exercising the rescaling freedom, we may always assume that the induced maps $\hoLie_2 \to \Graphs_2$ and $\hoLie_2\to \Br$ are the standard ones. 
By essentially the same argument one checks that a map of two colored operads $\ELie\to \op C$, where $\op C$ is a two colored operad quasi-isomorphic to $\bpm \Pois_2 & \Pois_2 & \Pois_2 \epm$, and which is an inclusion in homology is essentially unique, up to homotopy and rescalings.
Hence we may assume that -possibly after rescaling and passing to a homotopic morphism- we may ensure that the restriction of our morphism \eqref{equ:prestableformality} to $\ELie$ agrees with any given twistable stable formality morphism. For example, one such is obtained by operadic twisting from our given morphism. In particular it follows that the linear piece is a deformation of the HKR morphism, and thus the theorem is shown.
\hfill \qed

\bibliographystyle{plain}
\bibliography{../biblio}

\begin{thebibliography}{10}

\bibitem{AyalaFrancis}
David Ayala and John Francis.
\newblock Factorization homology of topological manifolds.
\newblock 2012.
\newblock arxiv:1206.5522.

\bibitem{ricardoBr}
Ricardo Campos.
\newblock {BV} formality, 2015.
\newblock in preparation.

\bibitem{CWtorsors}
Ricardo Campos and Thomas Willwacher.
\newblock { Operadic Torsors}.
\newblock 2014.
\newblock arXiv:1412.3614.

\bibitem{vasilystable}
Vasily Dolgushev.
\newblock {Stable Formality Quasi-isomorphisms for Hochschild Cochains I},
  2011.
\newblock arXiv:1109.6031.

\bibitem{vastwisting}
Vasily Dolgushev and Thomas Willwacher.
\newblock {Operadic twisting - with an application to Deligne's conjecture}.
\newblock {\em Journal of Pure and Applied Algebra}, 219(5):1349--1428, 2015.

\bibitem{GJ}
Ezra Getzler and J.~D.~S. Jones.
\newblock Operads, homotopy algebra and iterated integrals for double loop
  spaces, 1994.
\newblock arXiv:hep-th/9403055.

\bibitem{GoodwillieWeiss}
Thomas~G. Goodwillie and Michael Weiss.
\newblock Embeddings from the point of view of immersion theory. {II}.
\newblock {\em Geom. Topol.}, 3:103--118 (electronic), 1999.

\bibitem{HLTV}
Robert Hardt, Pascal Lambrechts, Victor Turchin, and Ismar Voli{\'c}.
\newblock Real homotopy theory of semi-algebraic sets.
\newblock {\em Algebr. Geom. Topol.}, 11(5):2477--2545, 2011.

\bibitem{K2}
Maxim Kontsevich.
\newblock Operads and {M}otives in {D}eformation {Q}uantization.
\newblock {\em Lett. Math. Phys.}, 48:35--72, 1999.

\bibitem{K1}
Maxim Kontsevich.
\newblock Deformation quantization of {P}oisson manifolds.
\newblock {\em Lett. Math. Phys.}, 66(3):157--216, 2003.

\bibitem{LV}
Pascal Lambrechts and Ismar Voli{\'c}.
\newblock Formality of the little {$N$}-disks operad.
\newblock {\em Mem. Amer. Math. Soc.}, 230(1079):viii+116, 2014.

\bibitem{livernet}
Muriel Livernet.
\newblock {Non-formality of the Swiss-Cheese operad}.
\newblock 2014.
\newblock arXiv:1404.2484.

\bibitem{lodayval}
J.-L. Loday and B.~Vallette.
\newblock {\em {Algebraic Operads}}.
\newblock Number 346 in {Grundlehren der mathematischen Wissenschaften}.
  {Springer, Heidelberg}, {2012}.

\bibitem{voronov_swisscheese}
Alexander~A. Voronov.
\newblock The {S}wiss-cheese operad.
\newblock In {\em Homotopy invariant algebraic structures ({B}altimore, {MD},
  1998)}, volume 239 of {\em Contemp. Math.}, pages 365--373. Amer. Math. Soc.,
  Providence, RI, 1999.

\bibitem{graphthings}
Thomas Willwacher.
\newblock {A Note on Br-infinity and KS-infinity formality}, 2011.
\newblock arXiv:1109.3520.

\bibitem{grt}
Thomas Willwacher.
\newblock M. {K}ontsevich's graph complex and the
  {G}rothendieck--{T}eichm\"uller {L}ie algebra.
\newblock {\em Invent. Math.}, 200(3):671--760, 2015.

\end{thebibliography}

\end{document}